\documentclass[reqno,10pt]{amsart}
\usepackage{amscd,amssymb,amsmath,color,url}

\usepackage[all]{xy}

\def\onto{\twoheadrightarrow}
\def\into{\hookrightarrow}

\def\rra{\rightrightarrows}

\def\toisom{\widetilde{\to}}
\def\otisom{\widetilde{\leftarrow}}

\def\.{,\dotsc ,}
\def\:{\colon}
\def\wt{\widetilde}

\def\ol{\overline}

\def\Ob{{\rm Ob}}

\def\QCoh{{\rm QCoh}}

\def\Spec{{\rm Spec}}

\def\Frac{{\rm Frac}}
\def\bfSpec{{\bf Spec}}

\def\aff{{\rm aff}}

\def\Ker{{\rm Ker}}

\def\Id{{\rm Id}}

\def\bfZ{{\bf Z}}

\def\gtC{{\mathfrak C}}
\def\gtD{{\mathfrak D}}

\def\bbZ{{\mathbb Z}}

\def\calA{{\mathcal A}}
\def\calB{{\mathcal B}}
\def\calC{{\mathcal C}}
\def\calD{{\mathcal D}}

\def\calF{{\mathcal F}}
\def\calG{{\mathcal G}}

\def\calK{{\mathcal K}}

\def\calM{{\mathcal M}}

\def\calO{{\mathcal O}}
\def\calP{{\mathcal P}}
\def\calQ{{\mathcal Q}}

\def\calU{{\mathcal U}}

\def\calW{{\mathcal W}}

\def\oT{{\ol T}}

\def\oY{{\ol Y}}
\def\oZ{{\ol Z}}

\def\og{{\ol g}}

\def\ox{{\ol x}}

\def\tilZ{{\wt Z}}

\def\tilcalP{{\wt\calP}}

\def\alp{{\alpha}}

\def\veps{\varepsilon}

\def\R+*{{\bf R^*_+}}

\def\colim{{\rm{colim}}}

\newtheorem{theor}[subsubsection]{Theorem}
\newtheorem{prop}[subsubsection]{Proposition}
\newtheorem{lem}[subsubsection]{Lemma}
\newtheorem{cor}[subsubsection]{Corollary}

\newtheorem{question}[subsubsection]{Question}

\theoremstyle{definition}

\newtheorem{rem}[subsubsection]{Remark}

\newtheorem{exam}[subsubsection]{Example}

\def\aff{\mathfrak{Aff}}
\def\Sch{\mathfrak{Sch}}
\def\s{\mathfrak{Sp}}

\def\colim{{\rm colim}}

\def\Coeq{{\rm Coeq}}

\begin{document}

\author{Michael Temkin, Ilya Tyomkin}
\thanks{Both authors were supported by the Israel Science Foundation (grant No. 1018/11). The second author was also partially supported by the European FP7 IRG grant 248826.}
\title{Ferrand pushouts for algebraic spaces}

\address{Einstein Institute of Mathematics, The Hebrew University of Jerusalem, Giv'at Ram, Jerusalem, 91904, Israel}
\email{temkin@math.huji.ac.il}
\address{Department of Mathematics, Ben-Gurion University of the Negev, P.O.Box 653, Be'er Sheva, 84105, Israel}
\email{tyomkin@math.bgu.ac.il}
\keywords{Ferrand pushouts, algebraic spaces.}
\begin{abstract}
We extend Ferrand's results about pushouts of schemes to the category of algebraic spaces. 
\end{abstract}

\maketitle

\section{Introduction}
In \cite{Fer}, D. Ferrand studied schematic pushouts of the form $Y\coprod_TZ$, where $f\:T\to Y$ is an affine morphism and $g\:T\into Z$ is a closed immersion. When $f$ is finite such pushout is called {\em pinching} or pinching of $Z$ with respect to $f$. Although studying pinchings was, probably, Ferrand's main motivation, he realized that the ``right generality", which allows one to prove all the fundamental results, is obtained by weakening the finiteness assumption on $f$. In the current paper we study the case of algebraic spaces $Y$, $Z$, and $T$ with the same assumptions on $f$ and $g$. We call such a triple $\calP=(T;Y,Z)$ a {\em Ferrand pushout datum}. If $\calP$ admits a pushout $X$ in the category of algebraic spaces such that the morphisms $Y\to X$ and $Z\to X$ are affine then $X=\coprod\calP=Y\coprod_TZ$ is called {\em Ferrand pushout}.

\subsection{Motivation}
Let $g\:T\into Z$ be a closed immersion. If $f\:T\to Y$ is also a closed immersion, then the pinching $X=Y\coprod_TZ$ can be viewed as the scheme obtained by gluing $Y$ and $Z$ along the closed subscheme $T$.

A more interesting and less intuitive case is the case when $f$ is an affine open immersion, or, more generally, a pro-open immersion; e.g., the embedding of the generic point. In this case we call the Ferrand pushout {\em composition}, and say that $X$ is obtained by composing $Y$ and $Z$ along $T$. At first glance, gluing an open subscheme of $Y$ to a closed subscheme of $Z$ may seem unnatural. For example, such pushout is usually non-noetherian even if $Y$ and $Z$ are. Nevertheless, compositions naturally appear in the theory of valued rings and schemes over them. For example, any valuation ring of non-zero finite height is composed of valuation rings of height one, and on the geometric side this corresponds to the composition of spectra of valuation rings of height one. More generally, compositions of schemes appear in applications of valuation theory to algebraic geometry, such as the study of relative Riemann-Zariski spaces and the proof of Nagata's compactification theorem by the first author \cite[\S2.3]{temrz}. Similarly, D. Rydh uses compositions in his work \cite[\S6]{Rydh} on Nagata's compactification of certain classes of algebraic stacks.

In the sequel papers \cite{pruf, nag} we define valuation algebraic spaces, and use them to study RZ spaces in the category of algebraic spaces, obtaining, as an application, a new proof of Nagata's compactification for algebraic spaces. Composing valuation algebraic spaces plays an important technical role in these papers. Although we only use compositions in our applications, and the proofs are slightly easier when $f$ is a monomorphism, we decided to study arbitrary Ferrand pushouts in the category of algebraic spaces because the main results hold true in this generality, and the additional arguments needed for this case are not very involved.

Curiously enough, the only published proof of Nagata's compactification for algebraic spaces does not involve compositions but makes a serious use of pinchings, see \cite[Theorem~2.2.2]{CLO}. Our results, in particular, subsume that theorem.

\subsection{Main results}

\subsubsection{Properties of Ferrand pushouts}
We say that a Ferrand pushout datum $\calP$ is {\em affine} if $Y,Z$ and $T$ are so. It is easy to see that if the Ferrand pushout $X=\coprod\calP$ is a scheme (resp. affine) then $\calP$ possesses an open affine covering $\calP=\cup_i\calP_i$ (resp. $\calP$ is affine). In Theorem~\ref{lem:affpushispush} we prove that, conversely, if $\calP$ possesses an open affine covering then there exists Ferrand pushout $X=\coprod\calP$, and if $\calP$ is affine then so is $X$. A closely related Theorem~\ref{schemth} asserts that the pushouts of schemes constructed by Ferrand in \cite{Fer} are, in fact, pushouts in the category of all algebraic spaces.

The following properties of Ferrand pushout $X=Y\coprod_TZ$ are established in Theorem~\ref{cor:ferprop}: (i) the pushout is compatible with topological realizations, i.e., $|X|=|Y|\coprod_{|T|}|Z|$, (ii) $T=Y\times_XZ$, (iii) set-theoretically, $X$ is the disjoint union of its closed subspace $Y$ and open subspace $Z\setminus T$.

Claim (i) is extended to other topologies in Theorem~\ref{equivcat}. Moreover, pullback and Ferrand pushout induce equivalences of categories of flat $X$-schemes and effective flat pushout data over $\calP$, which preserve various properties of morphisms.

Finally, by Theorem~\ref{ferdesth}, certain properties of morphisms hold for $Z\to X$ if and only if they hold for $T\to Y$; and by Theorem~\ref{Sproperties}, certain properties relative to a base space hold for $X=\coprod\calP$ if and only if they hold for $\calP$.

\subsubsection{Existence of Ferrand pushouts}
We prove in Theorem~\ref{effth} that the Ferrand pushout $\coprod\calP$ exists if and only if the pushout datum $\calP=(T;Y,Z)$ admits an affine \'etale covering $\{\calP_i\to\calP\}_i$, i.e., affine \'etale coverings $\{Y_i\to Y\}_i$, $\{Z_i\to Z\}_i$, and $\{T_i\to T\}_i$ with identifications $T\times_YY_i=T_i=T\times_ZZ_i$. Whether such coverings exist is an innocent-looking problem, which seems to be rather difficult. It is closely related to the lifting problem for \'etale morphisms $T'\to T$ with respect to a closed immersion $T\into Z$, see Section~\ref{affliftsec}. We managed to solve the latter problem for ind-quasi-affine $Z$, but the general case remains open. Consequently, we prove in Theorem~\ref{cor:prescor} that an \'etale affine covering exists in the following cases: (i) $Z$ is ind-quasi-affine, (ii) the pushout is a pinching, (iii) $Z$ is decent and $|T|$ is finite and discrete. We do not know a single example of Ferrand pushout datum that admits no affine \'etale covering.

\begin{rem}
(i) Theorem~\ref{cor:prescor}(i) is the main geometric input in our approach, which, in particular, is used to prove that affine Ferrand pushouts are pushouts in the category of all algebraic spaces (surjectivity of $\psi$ in the proof of Theorem~\ref{lem:affpushispush}).

(ii) Existence of pinchings in the non-noetherian case is a new result. Artin proved in \cite[Theorem 6.1]{Artin} that pinchings always exist in the noetherian case, see also \cite[Theorem~2.2.2]{CLO} for a more complete claim. Artin's proof uses his algebraization criteria, so the noetherian assumption cannot be eliminated. Note also that pinchings (and all Ferrand pushouts) are only compatible with flat base changes, so one cannot deduce the non-noetherian case by approximating $X=Y\coprod_TZ$ by noetherian algebraic spaces. A simpler proof of existence of pinchings, based on construction of an \'etale affine covering, was given by Kollar in \cite[Theorem~38]{Kollar}. Although the noetherian assumption is made in loc.cit. too, it can be easily removed. Our proof in the case of pinchings is a variation on Kollar's one.
\end{rem}

\subsection{Plan of the paper}
The paper is written using the language of pushout data and their morphisms, see \S\ref{defsec}. Preliminary results on general pushouts and affine Ferrand pushouts are collected in Sections \ref{prelimsec}--\ref{pushsec}. Most of this material is known, in particular, we recall Ferrand's results. In Section~\ref{genpushsec}, we start by constructing pushouts for $\calP$ admitting open affine coverings, see Lemma~\ref{Zarlem0} and Theorem~\ref{lem:affpushispush}, and then extend this to the case when $\calP$ possesses only an \'etale affine covering $\{\calP_i\to\calP\}$ such that each $\calP_i\times_\calP\calP_i$ possesses an open affine covering. The latter condition is removed in Theorem~\ref{effth} since it is always satisfied, but to prove this, we must study existence of \'etale affine coverings in Section~\ref{affpressec}. Finally, Section~\ref{lastsec} is devoted to proving finer criteria for existence of pushouts, see Theorem~\ref{effth}, and studying how various properties of spaces and morphisms descend through pushouts, see Theorems~\ref{equivcat}, \ref{ferdesth} and \ref{Sproperties}.

\subsubsection*{Acknowledgments}
We are grateful to Yakov Varshavsky for helpful discussions, and to an anonymous referee for valuable remarks, and for suggesting not to restrict the generality to the case of quasi-compact and quasi-separated algebraic spaces.

\setcounter{tocdepth}{1}

\section{Preliminaries}\label{prelimsec}

\subsection{Algebraic spaces}
We adopt the definitions, notation, and conventions of \cite{stacks}. In particular, we do not assume that algebraic spaces are quasi-separated. 

\subsubsection{Coverings and presentations}
By an {\em affine covering} of an algebraic space $X$ in a given topology $\tau$ we mean a $\tau$-covering $\{U_i\to X\}$ such that all $U_i$ are affine. 
For a $\tau$-covering $\{U_i\to X\}$, set $U:=\coprod U_i$ and $R:=U\times_XU$. Then $R\rightrightarrows U$ is a $\tau$-equivalence relation, and $X=U/R$. If $U$ is a scheme then so is $R$, and we say that $R\rightrightarrows U$ is a {\em $\tau$-presentation} of $X$.

\subsubsection{Ind-quasi-affine morphisms}
Recall that a scheme $X$ is called {\em ind-quasi-affine} if every quasi-compact open of $X$ is quasi-affine. Similarly, a morphism of schemes $X\to Y$ is called {\em ind-quasi-affine} if the pullback of an open affine is ind-quasi-affine. It is shown in \cite[Tag:0AP5]{stacks} that ind-quasi-affiness is stable under base change, fpqc local on the base, and that any separated locally quasi-finite morphism is ind-quasi-affine. We say that a morphism of algebraic spaces $X\to Y$ is {\em ind-quasi-affine} if for any morphism $Z\to Y$ with an affine source the pullback $X\times_YZ$ is ind-quasi-affine. The following follows easily from \cite[Tags:0AP5,02X4,0418]{stacks}:

\begin{lem}\label{compqaff}
(i) Any separated locally quasi-finite morphism is ind-quasi-affine. In particular, so is the diagonal $X\to X\times X$ of an algebraic space $X$.

(ii) If $Y$ is ind-quasi-affine then so is any morphism $Y\to X$.
\end{lem}

\subsubsection{Zariski points}\label{zarpointssec}
Recall that a {\em point} of an algebraic space $X$ is an equivalence class of morphisms $\Spec(F)\to X$, where $F$ is a field, and two maps are equivalent if and only if they are dominated by a third map. A point $x$ is called {\em Zariski point} if its class contains a quasi-compact monomorphism. Recall that $X$ is called {\em decent} in \cite[Tag:03I8]{stacks} if all its points are Zariski points; e.g., quasi-separated spaces are decent by \cite[Tag:03JX]{stacks}. It follows from \cite[Tag:0BBN]{stacks} that if $x$ is a Zariski point, and $f\:\Spec(F)\to X$ is its quasi-compact monomorphic representative then (1) any representative of $x$ factors through $f$, and (2) $f$ factors through an \'etale presentation. Thus, the quasi-compact monomorphic representative of $x$ is unique up-to an isomorphism, and by abuse of language, we will not distinguish between $x$ and $f\:\Spec(F)\to X$, and will call $F$ the {\em residue field of $x$}.

\subsubsection{A criterion for being a monomorphism}

\begin{lem}\label{lem:monomor}
Assume that $h\:Y\to X$ is a representable morphism locally of finite type and $g\:X'\to X$ is any surjective morphism. Then $h$ is a monomorphism if and only if so is $h\times_XX'$.
\end{lem}
\begin{proof}
Monomorphisms are stable under arbitrary base changes, thus one direction is clear. By \cite[Tag:042Q]{stacks}, we may assume that $X$ and $X'$ are schemes, and hence so is $Y$. For any point $x'\in X'$ the base change $h\times_X\Spec(k(x'))$ is a monomorphism. It then follows by fpqc descent that the same is true for $h\times_X\Spec(k(x))$ for $x=g(x')$. Since $g$ is surjective, the restriction of $h$ over any point of $X$ is a monomorphism, hence $h$ is a monomorphism by \cite[$\rm IV_4$, 17.2.6]{ega}.
\end{proof}

\subsection{General pushouts}

\subsubsection{Pushout data}\label{defsec}
Let $\gtC$ be one of the following categories: affine schemes, schemes, or algebraic spaces. By a {\em pushout datum} $\calP$ in $\gtC$ we mean a diagram $Y\stackrel{f}\longleftarrow T\stackrel{g}\longrightarrow Z$, where $T,Y,Z\in\Ob(\gtC)$, and $f,g$ are {\em separated} morphisms. The colimit of $\calP$ in $\gtC$ is called the {\em pushout} of $Y$ and $Z$ with respect to $T$, and is denoted by $\coprod^\gtC\calP$ or $Y\coprod_T^\gtC Z$. By abuse of natation, we often refer to $\calP$ as triple $(T;Y,Z)$, omitting the morphisms in the notation. If $\gtC$ is our default category of algebraic spaces we omit $\gtC$ in the notation $\coprod^\gtC$.
\begin{rem}\label{subtlepushrem}
It often happens that two geometric categories $\gtC\subset\gtC'$ possess different pushouts $X=Y\coprod^\gtC_TZ$ and $X'=Y\coprod^{\gtC'}_TZ$, i.e., the natural $\gtC'$-morphism $X'\to X$ is not an isomorphism. 
A classical example of this phenomenon is a quotient stack $X'$ and its coarse moduli space $X$. Another example is the pushout of $\Spec(k[x])$ and $\Spec(k[x^{-1}])$ along $\Spec(k[x,x^{-1}])$. While in the category of $k$-schemes the pushout is the projective line, in the category of $k$-affine schemes it is just a point. We will see that Ferrand pushouts are more stable, and, in particular, an affine Ferrand pushout is the pushout in the category of all algebraic spaces.
\end{rem}

Consider a commutative diagram of pushout data $\calP$ and $\calP'$
$$
\xymatrix{
Y\ar[d]^{\phi_Y} & T\ar[l]_{f}\ar[r]^{g}\ar[d]^{\phi_T}& Z\ar[d]^{\phi_Z} &\\
Y' & T'\ar[l]_{f'}\ar[r]^{g'} & Z'
}
$$
If both squares are cartesian then $\phi\:\calP\to\calP'$ is called a {\em morphism of pushout data}. If $Y'=T'=Z'=U$ and $f'=g'=id_U$ then $\phi\:\calP\to U$ is called a {\em map} from the pushout datum $\calP$ to the algebraic space $U$.
Plainly, a composition of morphisms is a morphism, and the composition of a morphism and a map is a map.

The cartesian condition is rather restrictive. For instance, if $\calP\times\calP$ is the product defined componentwise, then the natural projections $\calP\times\calP\rightrightarrows \calP$ are usually not morphisms. However, one easily checks that the category of pushout data admits {\em fiber products}, which are nothing but componentwise fiber products. Moreover, if $\calP\to U$ is a map and $U'\to U$ is a morphism of algebraic spaces then $\calP\times_UU'\to \calP$ is a morphism, where $\calP\times_UU'$ is defined componentwise.

Unless explicitly said to the contrary, we say that a diagram of pushout data (resp. a pushout datum) possesses certain property if all its components do. In particular, if $\gtC$ is the category of algebraic spaces then we define affine pushout data, \'etale morphisms, fppf affine covers, etc., via this rule.

\begin{rem}
Since the category of pushout data admits no products one should be careful while defining equivalence relation. So, by an fppf or \'etale {\em equivalence relation} we mean two morphisms $\calP_1\rra\calP_0$ defining fppf or \'etale equivalence relations on the $T_i$, $Y_i$, and $Z_i$ components.
\end{rem}

\begin{lem}
Let $\calP_1\rra\calP_0$ be an fppf equivalence relation. Set $Y:=Y_0/Y_1$, $Z:=Z_0/Z_1$, $T:=T_0/T_1$, and $\calP:=\calP_0/\calP_1=(T;Y,Z)$. Then $Y,Z,T$ are algebraic spaces, the commutative diagram $\calP_0\to\calP$ is a morphism, and $\calP_1=\calP_0\times_\calP\calP_0$.
\end{lem}
\begin{proof}
By Artin's theorem \cite[Tag:04S6]{stacks} $Y$, $Z$, and $T$ are algebraic spaces. Thus, the last two assertions follow from \cite[Tag:07S3]{stacks}.
\end{proof}

\subsubsection{Quasi-coherent modules, pullback and pushforward functors}
For any pushout datum $\calP=(T;Y,Z)$ set $\calO_\calP:=(\calO_T;\calO_Y,\calO_Z)$. A {\em quasi-coherent $\calO_\calP$-module} $\calF$ is a triple of quasi-coherent modules $(\calF_T;\calF_Y,\calF_Z)$ with isomorphisms $\alpha_\calF\:f^*(\calF_Y)\to\calF_T$ and $\beta_\calF\:g^*(\calF_Z)\to\calF_T$. This agrees with the usual definition of modules over algebraic spaces if $Y=T=Z$. An $\calO_\calP$-module is {\em flat}, {\em finitely generated}, etc., if all its components are. A morphism of modules is a triple of morphism compatible with the homomorphisms $\alpha_\bullet$ and $\beta_\bullet$. This defines a category $\QCoh_\calP$.

For a map $\phi\:\calP\to U$ from a pushout datum to an algebraic space we have the pullback functor $\phi^*\:\QCoh_U\to\QCoh_{\calP}$ defined componentwise. 
\begin{lem}\label{phi*}
Let $\phi\:(T;Y,Z)\to U$ be a map from a pushout datum to an algebraic space. Then $\phi^*$ admits right adjoint $\phi_*$, which is given by the formula
$$\phi_*(\calF_T;\calF_Y,\calF_Z)=\phi_{Y*}(\calF_Y)\times_{\phi_{T*}(\calF_T)}\phi_{Z*}(\calF_Z).$$
\end{lem}
\begin{proof}
Pick $\calG\in\QCoh_U$. To give a homomorphism $\phi^*(\calG)\to\calF$ is the same as to give compatible homomorphisms from $\calG$ to $\phi_{Y*}(\calF_Y)$, $\phi_{Z*}(\calF_Z)$ and $\phi_{T*}(\calF_T)$, which is equivalent to giving a homomorphism $\calG\to\phi_{Y*}(\calF_Y)\times_{\phi_{T*}(\calF_T)}\phi_{Z*}(\calF_Z)$.
\end{proof}
\begin{rem}
A similar theory applies to the categories of quasi-coherent algebras verbatim, so we omit the details.
\end{rem}


\section{Affine Ferrand pushouts}\label{pushsec}

\subsection{Terminology}
A pushout datum $\calP=(T;Y,Z)$ is called {\em Ferrand} if $T\to Y$ is affine and $T\to Z$ is a closed immersion. In the category of schemes such pushouts were introduced and studied extensively by D. Ferrand \cite{Fer}. If the pushout $X=\coprod\calP$ exists and the morphisms $Y\to X$ and $Z\to X$ are affine then we say that $X$ is a {\em Ferrand pushout} and the Ferrand pushout datum $\calP$ is {\em effective}.

If all components of a Ferrand pushout datum $\calP$ are affine then we say that $\calP$ is an {\em affine Ferrand pushout datum} and $X=\coprod^{\aff}\calP$ is the {\em affine Ferrand pushout}. We will prove in Theorem~\ref{lem:affpushispush} that $X=\coprod\calP$, so it is, in fact, Ferrand pushout of $\calP$, but we have to distinguish the two notions until then. 

\subsection{Ferrand diagrams of rings}\label{ringssec}

\subsubsection{The definition}\label{ferdefsec}
Let $B\to K\leftarrow C$ be homomorphisms of rings such that $C\to K$ is surjective. Set $A:=B\times_KC$, and consider the commutative diagrams
\begin{equation}\label{eq:ferdiag}
\xymatrix{
K & C\ar@{->>}[l] & & & T\ar[d]\ar@{^(->}[r] & Z\ar[d]^{\Spec(\phi)}\\
B\ar[u] & A\ar@{->>}[l]\ar[u]_\phi & & & Y\ar@{^{(}->}[r] & X
}
\end{equation}
where the right diagram is the diagram of the corresponding spectra. We say that the left cartesian square is a {\em Ferrand diagram} of rings.

\subsubsection{Conductor}
The ideal $I:=\Ker(A\onto B)$ is called the {\em conductor} of the Ferrand diagram. Note that $\phi$ induces an isomorphism $I\to \Ker(C\onto K)$, and the diagram is completely determined by $\phi$ and $I$. Vice versa, for a homomorphism $\phi\:A\to C$ and an ideal $I\subset A$ such that $I\toisom\phi(I)$ is an ideal of $C$, set $B:=A/I,K:=C/I$. Then the corresponding diagram is Ferrand.

\subsubsection{Basic properties}
By a bicartesian square we mean a square diagram which is both cartesian and cocartesian.

\begin{lem}\label{firstproplem}
Let assumptions and notation be as in diagram  \eqref{eq:ferdiag}. Then,

(i) $B\otimes_AC=K$, so Ferrand diagrams are bicartesian.

(ii) If $A\to A'$ is flat then the diagram \eqref{eq:ferdiag}$\otimes_AA'$ is Ferrand.

(iii) $\Spec(\phi)$ induces an isomorphism of open subschemes $Z\setminus T=X\setminus Y$.
\end{lem}
\begin{proof}
(i) Let $I$ be the conductor. Then $B\otimes_AC=C/IC=C/I=K$.

(ii) The sequence $0\to A\to B\oplus C\to K$ is exact, hence by the flatness assumption so is $0\to A'\to B'\oplus C'\to K'$, where $\bullet':=\bullet\otimes_AA'$. Thus, $A'=B'\times_{K'}C'$, and the diagram \eqref{eq:ferdiag}$\otimes_AA'$ is Ferrand.

(iii) By (ii), $A_f=B_f\times_{K_f}C_f$ for any $f\in A$. Thus, for $f\in I$, $B_f=K_f=0$, $A_f=C_f$, and so $Z_f=X_f$. Hence $Z\setminus T=\cup_{f\in I}Z_f=\cup_{f\in I}X_f=X\setminus Y$.
\end{proof}

\begin{exam}\label{ex:afpo}
If $K=k[x^{\pm 1}]$, $B=k[x]$, $C=k[x^{\pm 1},y]$, and the map $C\to K$ sends $y$ to $0$, then $A=\{f\in k[x^{\pm 1},y]\,:\, f(x,0)\in k[x]\}$, and the conductor ideal is $I=yk[x^{\pm 1}]$. Notice that $A$ is not Noetherian even though $B,C,$ and $K$ are so.
\end{exam}

\subsection{Affine Ferrand pushouts}\label{affpushsec}

\subsubsection{Basic properties}\label{fersec}
In general, affine pushouts $\coprod^{\aff}\calP$ may contain almost no information about the components of $\calP$, cf. Remark~\ref{subtlepushrem}. However, affine Ferrand pushouts behave much better:

\begin{prop}\label{ferprop}
Let $X=Y\coprod^{\aff}_TZ$ be an affine Ferrand pushout. Then,

(i) The topological pushout $|Y|\coprod_{|T|}|Z|$ is naturally homeomorphic to $|X|$,

(ii) $T=Y\times_XZ$,

(iii) $Y\to X$ is a closed immersion, $U=Z\setminus T\to X$ is an open immersion, and $|X|=|Y|\coprod|U|$ set-theoretically.
\end{prop}
\begin{proof}
Assertion (i) follows from \cite[Theorem~5.1 and Scholie~4.3]{Fer}, and assertions (ii)-(iii) follow from Lemma~\ref{firstproplem}.
\end{proof}

\subsubsection{Liftings of semivaluations}
Assume that $T\to Y$ is an open immersion, in particular, the pushout is a composition. Then Proposition~\ref{ferprop}(i) implies that the topological space $|X|$ is glued from its open subspace $|Z|$ and closed subspace $|Y|$ along $|T|$. In particular, for any $z\in |Z|$ with a specialization $y\in |Y|$ there exists a point $t\in |T|$ with $y\preceq t\preceq z$. Thus, any continuous map from a topological space totally ordered by the specialization relation to $|X|$ for which the preimage of $|Y|$ consists of at most one point factors through $|Z|$. For example, in the situation of Example~\ref{ex:afpo}, a map $f\:\Spec(R)\to X$ from a valuation scheme factors through $Z$ if and only if $f^*x$ is invertible. If $f^*x$ is not invertible then the valuation of $f^*x$ is strictly positive, and since the valuation of $f^*(x^{-n}y)$ is non-negative, the valuation of $f^*y$ is strictly greater than the valuation of $f^*x^n$ for all $n$. Thus, $\sqrt{f^*I}=\sqrt{(f^*x^ny)_{n\in\bbZ}}\varsubsetneq (f^*x)\subseteq m_R$, and since radical ideals in valuation rings are prime we obtain that $f^{-1}(Y)$ contains at least two points. The following strengthening of this fact will be used to show that Ferrand pushouts preserve separatedness.

\begin{lem}\label{vallem}
Let assumptions and notation be as in diagram \eqref{eq:ferdiag}. Assume that $R$ is a valuation ring of non-zero height, i.e., $R$ is not a field, $S:=\Spec(R)$, $s\in S$ the closed point, and $f\:S\to X$ a morphism such that $f^{-1}(Y)=\{s\}$. Then $f$ admits a unique lifting $g\:S\to Z$, and the latter morphism satisfies $g^{-1}(T)=\{s\}$.
\end{lem}
\begin{proof}
The uniqueness of lifting follows from the separatedness of $Z\to X$, and $g^{-1}(T)=f^{-1}(Y)=\{s\}$ because $Z\setminus T=X\setminus Y$. Thus, let us prove the existence.

Let $\eta\in X$ be the image of the generic point of $S$. Then $k(\eta)\subseteq\Frac(R)$ and $R'=R\cap k(\eta)$ is a valuation ring. Since $f$ factors through a morphism $f'\:\Spec(R')\to X$, it suffices to lift $f'$ to $Z$. Thus, after replacing $R$ with $R'$ we may assume that $R$ is a valuation ring of $k(\eta)$. By our assumptions $\eta\in Z\setminus T$, hence the homomorphism $A\to k(\eta)$ factors through the homomorphism $C\to k(\eta)$, and we should only prove that the image $C(\eta)\subset k(\eta)$ of $C$ lies in $R$.

Assume to the contrary that some $x\in C$ is mapped to $\ox=x(\eta)\in k(\eta)\setminus R$. Then $\ox^{-1}$ lies in the maximal ideal $m_R$. Let $I\subset A$ be the conductor of the diagram. Since $f^{-1}(Y)=\{s\}$, the radical of $IR$ coincides with $m_R$. In particular, there exists $n\ge 1$ such that $\ox^{-n}\in IR$, and hence $\ox=\ox^{n+1}\ox^{-n}\in CIR=IR\subset R$, which is a contradiction.
\end{proof}

Recall that a ring $A$ is Pr\"ufer if and only if all its localizations are valuation rings. So, the following is an immediate consequence of Lemma~\ref{vallem}.

\begin{cor}\label{valcor}
Let assumptions and notation be as in diagram \eqref{eq:ferdiag}. Assume that $R$ is a Pr\"ufer ring such that $S:=\Spec(R)$ has no isolated points, and $f\:S\to X$ is a morphism such that the preimage of $Y$ is contained in the set of closed points of $S$. Then $f$ admits a unique lifting to a morphism $g\:S\to Z$.
\end{cor}

\subsubsection{Descent of properties through $\phi_*$}\label{phisec}
Assume that $X=Y\coprod^{\aff}_TZ$ is a Ferrand affine pushout, and let $\phi\:\calP\to X$ be the natural map. 
While the pullback preserves all natural properties of modules, the situation with the pushforward is more delicate. When working on \cite{temrz} the first author learned from D. Rydh the following example, in which finite presentation is not respected by $\phi_*$.

\begin{exam}\label{fpexam}
Let $k$ be a field. Consider discrete valuation rings $B=k[x]_{(x)}$ and $C=K[y]_{(y)}$, where $K=k(x)$. In particular, $C/yC=K=\Frac(B)$. Then their composition $A$ is the preimage of $B$ in $C$, which is a valuation ring of height two. In fact, $A$ is the valuation ring of $k(x,y)$ such that $|y|\ll|x|\ll 1$, where the valuation is written multiplicatively. Note that $C=A_x$ and $B=A/I$, where $I=yC$ is the conductor of the corresponding Ferrand diagram.

Set $I_n:=x^{-n}yA$, and note that $I=\cup I_n$ is not finitely generated as an ideal of $A$ because $I_n\varsubsetneq I_{n+1}$ for any $n$. In particular, $A'=A/I=B$ is finitely generated but not finitely presented over $A$, while $A'_n=A/I_n$ are finitely presented. Note that the $B$-module $B'_n:=A'_n\otimes_AB=B$, the $C$-module $C'_n:=A'_n\otimes_AC=C/I_nC=C/I=K$, and the $K$-module $K'_n=A'_n\otimes_AK=K$ are finitely presented and independent of $n$. In particular, $B'_n\times_{K'_n}C'_n=B'_n=A'$, and we obtain the following pathologies: $\phi_*$ takes the finitely presented $\calO_\calP$-module $(K'_n;B'_n,C'_n)$ to the non-finitely presented $A$-module $A'$, and the natural homomorphism $A'_n\to\phi_*\phi^*(A'_n)=A'$ is not injective.
\end{exam}

On the positive side, we have the following list of properties respected by $\phi_*$:

\begin{prop}\label{phidescentrop}
Let $\phi\:\calP\to X$ be as in \S\ref{phisec}, and let $(\dag)$ be one of the following properties of quasi-coherent modules (resp. algebras): (i) flat, (ii) finitely generated, (iii) flat and finitely presented. Then, both $\phi_*$ and $\phi^*$ preserve $(\dag)$.
\end{prop}
\begin{proof}
For modules, the assertion was proved by Ferrand \cite[Theorem~2.2(iv)]{Fer}] (one uses the fact that a flat finitely presented module is nothing but a projective module of finite type). Obviously, this also covers (i) for algebras.

To prove (ii) and (iii) for algebras, let $A'$ be an $A$-algebra, $B':=B\otimes_AA'$, and $C':=C\otimes_AA'$. In (ii) we should prove that if $B'$ and $C'$ are finitely generated then $A'$ is finitely generated. Obviously, there exists a finitely generated $A$-subalgebra $A''\subset A'$ such that the homomorphisms $B\otimes_AA''\onto B'$ and $C\otimes_AA''\onto C'$ are surjective. Since $\phi^*$ is right exact, this implies that $\phi^*(A'/A'')=0$, and then $A''=A'$ by \cite[Theorem~2.2(ii)]{Fer}].

Finally, assume that $B'$ and $C'$ are flat and finitely presented. Then $A'$ is $A$-flat by (i), and finitely generated over $A$ by (ii). It remains to prove that it is finitely presented over $A$. Pick a surjective $A$-homomorphism $f_A\:A''=A[T_1,\dotsc,T_n]\to A'$, and let us prove that $I_A=\Ker(f_A)$ is finitely generated. By (ii) for modules, it suffices to show that $I_B=I_A\otimes_AB$ and $I_C=I_A\otimes_AC$ are finitely generated. Since $A'$ is $A$-flat, $I_A$ is flat too. Therefore, $I_B=\Ker(f_B)$ and $I_C=\Ker(f_C)$, where $f_B\:B''\to B'$ and $f_C\:C''\to C'$ are the base changes of $f_A$. Thus, $I_B$ and $I_C$ are finitely generated modules since $B'$ and $C'$ are finitely presented.
\end{proof}

\subsubsection{The adjunctions}

\begin{prop}\label{adjunctprop}
Let $\phi\:\calP\to X$ be as in \S\ref{phisec}. Then,

(i) The adjunction $\phi^*\phi_*(\calM)\to\calM$ is an isomorphism for any quasi-coherent $\calO_\calP$-module (resp. $\calO_\calP$-algebra) $\calM$.

(ii) If $\calF$ is a flat quasi-coherent $\calO_X$-module (resp. $\calO_X$-algebra) then the adjunction $\calF\to\phi_*\phi^*(\calF)$ is an isomorphism.
\end{prop}
\begin{proof}
Obviously, it suffices to prove the proposition for modules, and in this case (i) is proved in \cite[Theorem~2.2(i)]{Fer}. As for (ii), one should check that any flat $A$-module $M$ satisfies $M=(M\otimes_AB)\times_{(M\otimes_AK)}(M\otimes_AC)$, and this is done by tensoring the exact sequence $0\to A\to B\oplus C\to K$ with $M$.
\end{proof}

Propositions~\ref{phidescentrop} and \ref{adjunctprop} immediately imply the following result.

\begin{cor}\label{fercor}
The functors $\phi_*$ and $\phi^*$ establish essentially inverse equivalences between the categories of flat modules (resp. algebras) over $\calO_X$ and $\calO_\calP$. Moreover, these functors establish equivalences of the full subcategories consisting of flat and finitely generated modules (resp. algebras), and of flat and finitely presented modules (resp. algebras).
\end{cor}

\begin{rem}
One cannot remove the flatness assumption in Corollary~\ref{fercor}. For example, in the situation described in Example~\ref{fpexam}, the finitely presented $A$-module $A'_n$ is not contained in the essential image of $\phi_*$. Indeed, if $A'_n\toisom\phi_*(M)$ for a $\calP$-module $M$ then $\phi^*(A'_n)\toisom\phi^*\phi_*(M)=M$ by Proposition~\ref{adjunctprop}(i), and so $\phi_*\phi^*(A'_n)\toisom\phi_*(M)\otisom A'_n$. But we saw in Example~\ref{fpexam} that the cyclic modules $\phi_*\phi^*(A'_n)=A'$ and $A'_n$ are not isomorphic.
\end{rem}

\subsection{Equivalence of categories}
Assume that $X=\coprod^{\aff}\calP$ and $\phi\:\calP\to X$ is the natural map. Then by $\phi^{-1}$ we denote the pullback functors from the category of $X$-spaces to the category of pushout data over $\calP$, i.e., $\phi^{-1}(X')=\calP\times_XX'$.

\begin{lem}\label{affproperties}
If $\calP$ is an affine Ferrand pushout datum, $X=\coprod^{\aff}\calP$, and $\gtC$ and $\gtD$ the categories of affine flat schemes (resp. pushout data) over $X$ (resp. $\calP$) then,

(i) $\phi^{-1}\:\gtC\to\gtD$ and $\coprod^{\aff}\:\gtD\to\gtC$ are essentially inverse equivalences.

(ii) Both $\phi^{-1}$ and $\coprod^\aff$ preserve the following properties of morphisms: (1) surjective, (2) \'etale, (3) flat, (4) finite type, (5) flat and finitely presented.
\end{lem}
\begin{proof}
(i) follows from Corollary~\ref{fercor} by passing to spectra since $\phi^{-1}$ corresponds to $\phi^*$ and $\coprod^{\aff}$ corresponds to $\phi_*$ by Lemma~\ref{phi*}.

(ii) Clearly, $\phi^{-1}$ respects all five properties, so let us prove that they are also respected by $\coprod^{\aff}$. Cases (3)--(5) follow from Corollary~\ref{fercor}. Recall that a finitely presented flat morphism is \'etale if and only if so are its fibers. Thus, case (2) follows from (5), since the morphism $h\:Y\coprod Z\to X$ is surjective by Proposition~\ref{ferprop}. In the same fashion, case (1) follows from the surjectivity of $h$.
\end{proof}

\section{General Ferrand pushouts}\label{genpushsec}
Now let us study general Ferrand pushouts. At one place we will have to use a result on existence of \'etale affine coverings of pushouts, whose proof is postponed  until Section~\ref{affpressec} for expositional reasons.

\subsection{Zariski globalization of $\coprod^\aff$}
Our first goal is to globalize the functor $\coprod^\aff$ with respect to Zariski topology.

\begin{lem}\label{Zarlem0}
There exists a unique extension of $\coprod^\aff$ to a functor $\coprod^\Sch$ that associates a scheme to a Ferrand pushout datum possessing an open affine covering, and takes open coverings to open coverings.
\end{lem}
\begin{proof}
Let $\calP$ be a Ferrand pushout datum with an open affine covering $\{\calP_i\}_{i\in I}$. We claim that each intersection $\calP_{ij}$ possesses an open affine covering $\{\calP_{ijk}\}_{k\in K_{ij}}$. Indeed, $\calP_{ij}$ is open in $\calP_i$, hence by Proposition~\ref{ferprop}(i) it is the pullback of an open subscheme $X_{ij}$ of $X_i=\coprod^\aff\calP_i$. Take an open affine covering $\{X_{ijk}\}_k$ of $X_{ij}$ and define $\calP_{ijk}$ to be its pullback to $\calP_{ij}$.

At this stage, uniqueness of $\coprod^\Sch$ is clear: for each $\calP_{ij}$ fix an open affine covering $\{\calP_{ijk}\}_{k\in K_{ij}}$ of $\calP_{ij}$ and glue $X=\coprod^\Sch\calP$ form the affine schemes $X_i$ along the open subschemes $X_{ij}=\cup_{k\in K_{ij}}\coprod^\aff\calP_{ijk}$. To prove existence one should check independence of the construction of choices. In fact, it suffices to check that one can replace the affine covering $\{\calP_i\}_i$ with its refinement, which is clear.
\end{proof}

Ferrand proved that $\coprod^\aff\calP$ is the pushout of $\calP$ in the category of schemes, which implies that so is $\coprod^\Sch\calP$; thereby justifying the notation. We shall not use this, and instead we we will prove directly that $\coprod^\Sch\calP$ is the pushout in the category of all algebraic spaces. Let us start with a necessary auxiliary result.

\begin{lem}\label{Zarlem}
Assume that $\calP$ is a Ferrand pushout datum that admits an open affine covering and $X=\coprod^\Sch\calP$. Let $\gtC$ be the category of flat schemes over $X$ (with arbitrary morphisms between them), and $\gtD$ the category of flat Ferrand pushout data over $\calP$ that possess an open affine covering. Then,

(i) The natural map $\phi\:\calP\to X$ is affine, $T=Y\times_XZ$, $Y\to X$ is a closed immersion, $U=Z\setminus T\to X$ is an open immersion, $|Y|\coprod_T|Z|=|X|$, and $|X|=|Y|\coprod|U|$ set-theoretically.

(ii) The functors $\phi^{-1}\:\gtC\to\gtD$ and $\coprod^\Sch\:\gtD\to\gtC$ are essentially inverse equivalences. In particular, they preserve products, and $\gtD$ admits finite products.

(iii) Both $\phi^{-1}\:\gtC\to\gtD$ and $\coprod^\Sch\:\gtD\to\gtC$ preserve the following properties of morphisms: (1) surjective, (2) \'etale, (3) flat, (4) locally of finite type, (5) flat and locally of finite presentation.
\end{lem}
\begin{proof}
(i) Let $\{\calP_i\}_i$ be an open affine covering of $\calP$. Set $X_i:=\coprod^\aff\calP_i$. Then, by the construction of $\coprod^\Sch$, the map $\calP\to X$ is glued on the base from the affine maps $\calP_i\to X_i$. The other claims now follow from Proposition~\ref{ferprop}.

(ii) We shall prove that $\coprod^\Sch(\calP\times_XX')=X'$ for an $X$-flat scheme $X'$ and $\calP'=(\coprod^\Sch\calP')\times_XX'$ for an appropriate $\calP$-pushout datum. Both functors are compatible with localizations, hence it suffices to consider the case when $X,X',\calP$ and $\calP'$ are affine, but this case is covered by Lemma~\ref{affproperties}(i).

(iii) Properties (2)--(5) are Zariski local on the target and on the source, hence the assertion follows from Lemma~\ref{affproperties}(ii). Finally, the claim about surjectivity follows from the surjectivity of $h\:Y\coprod Z\to X$.
\end{proof}

\subsection{The schematic case}\label{affinesec}
In the proof of the following theorem, given an algebraic space $U$ and a pushout datum $\calP$, denote by $h_U(\calP)$ the set of maps $\calP\to U$. Thus, a morphism $\calP\to X$ represents the pushout of $\calP$ if and only if the natural map $\psi\:h_U(X)\to h_U(\calP)$ is bijective for any $U\in Ob(\s)$.

\begin{theor}\label{lem:affpushispush}
Let $\calP$ be a Ferrand pushout datum that admits an open affine covering and $X=\coprod^\Sch\calP$. Then $\calP$ is effective and $X=\coprod\calP$.
\end{theor}
\begin{proof}
It suffices to show that $X=\coprod\calP$ since the morphism $\calP\to X$ is affine by Lemma~\ref{Zarlem}(i). First, let us reduce to the case when $\calP$ is affine. If $\{\calP_i\}_{i\in I}$ is an open affine covering of $\calP$ then we showed in the proof of Lemma~\ref{Zarlem0} that each $\calP_{ij}$ possesses an open affine covering $\{\calP_{ijk}\}_{k\in K_{ij}}$ and $X$ is glued from $X_i=\coprod^\aff\calP_i$ along the unions of $X_{ijk}=\coprod^\aff\calP_{ijk}$. In other words, $\calP=\colim_\calD\calP_*$ and $X=\colim_\calD X_*$, where $\calD$ denotes the diagram of arrows $(ijk)\to l$ with $i,j\in I$, $k\in K_{ij}$ and $l\in\{i,j\}$. Since colimits commute, if the theorem holds for affine pushouts then $\coprod\calP=\coprod\colim_\calD\calP_*=\colim_\calD\coprod\calP_*=\colim_D X_*=X.$

Assume now, that $\calP$ is affine, and hence $X=\coprod^\aff\calP$. we shall show that the natural map $\psi\:h_U(X)\to h_U(\calP)$ is bijective for any algebraic space $U$. Since $X,Y,Z,T$ are affine, $h_U(X)=\lim_W h_W(X)$ and $h_U(\calP)=\lim_W h_W(\calP)$, where the limits are taken over all open quasi-compact subsets $W\subseteq U$. Thus, we may assume that $U$ is quasi-compact. Fix an \'etale covering $U_0\to U$ with an affine source.

Injectivity of $\psi$: Pick $h_1,h_2\in h_U(X)$ and assume that $\psi(h_1)=\psi(h_2)$. Set $\calP_0':=\calP\times_UU_0$ and $X_{0i}':=X\times_{h_i,U}U_0$. Since $X'_{0i}$ are schemes and $X_{01}'\times_X\calP\simeq\calP_0'\simeq X_{02}'\times_X\calP$, it follows from Lemma~\ref{Zarlem}(ii) that $X_{01}'\simeq X_{02}'$, which we denote simply by $X_0'$. Since $X'_0$ is quasi-compact there exists a surjective \'etale morphism $X_0\to X_0'$ with an affine source. Set $\calP_0:=\calP\times_XX_0=\calP_0'\times_{X_0'}X_0$, then $X_0=\coprod^\aff\calP_0$ by Lemma~\ref{Zarlem}(ii). Consider the commutative diagram
$$
\xymatrix{
\calP_0\ar@{->>}[rr]\ar[d]&& \calP_0'\ar@{->>}[rr]\ar[rdd]\ar[d]&& \calP\ar[d]\ar[rdd]^(.35){\psi(h_1)=\psi(h_2)}\\
X_0\ar@{->>}[rr] && X_0'\ar@<0.5ex>[rd]\ar@<-0.5ex>[rd]\ar@{->>}[rr]|!{[u];[dr]}\hole && X\ar@<0.5ex>[rd]_(.4)h_1\ar@<-0.5ex>[rd]^(.4){h_2}\\
&&& U_0\ar@{->>}[rr]&& U
}
$$
It is enough to prove that the base change morphisms $X'_0\rightrightarrows U_0$ coincide, or, equivalently, that the two composed morphisms $X_0\to X'_0\rra U_0$ coincide. But, the composed map $\calP_0\to\calP'_0\to U_0$ factors uniquely through $X_0=\coprod^\aff\calP_0$ since $U_0$ is affine, and hence $h_1=h_2$.

Surjectivity of $\psi$: Let $\alp\in h_U(\calP)$ be arbitrary, and set $\calP_0':=\calP\times_UU_0$. Then $Z'_0$ is an ind-quasi-affine scheme and hence $\calP'_0$ possesses an \'etale affine covering $\calP_0\to\calP'_0$ by Corollary~\ref{prescor1}. Set $\calP_1=\calP_0\times_\calP\calP_0$ and $X_i:=\coprod^\aff\calP_i$ for $i=0,1$, then $\calP_0\to U_0$ factors through a morphism $h_0\:X_0\to U_0$. Note that $X_1\rra X_0$ is an \'etale equivalence relation with quotient $X$ because $X_0\to X$ is an \'etale covering by Lemma~\ref{Zarlem}(iii) and $X_1=X_0\times_XX_0$ by Lemma~\ref{Zarlem}(ii). Consider the commutative diagram

$$
\xymatrix{
\calP_1\ar@<0.5ex>[rr]\ar@<-0.5ex>[rr]\ar[d]&& \calP_0\ar@{->>}[rr]^\pi\ar[d]\ar[rrdd]&& \calP\ar[d]_{\phi}\ar[rdd]^(.35)\alp\\
X_1\ar@<0.5ex>[rr]\ar@<-0.5ex>[rr]&& X_0\ar@{->>}[rr]|!{[u];[drr]}\hole\ar[rrd]_{h_0}&& X\ar@{-->}[rd]_(.4)h\\
&&&& U_0\ar@{->>}[r]& U
}
$$

The compositions of $\calP_1\to X_1$ with the two composed morphisms $X_1\rightrightarrows X_0\to U_0\to U$ coincide. Thus, the two morphisms $X_1\rightrightarrows U$ coincide by the injectivity of $\psi_1\:h_U(X_1)\to h_U(\calP_1)$. Hence the morphism $X_0\to U$ induces a morphism $h\:X\to U$ since $X=X_0/X_1$. It remains to verify that $\psi(h)=\alpha$. By the construction $\alpha\circ\pi=h\circ\phi\circ\pi$, where $\pi\:\calP_0\to\calP$ is the \'etale covering. Thus $\alpha=h\circ\phi=\psi(h)$, and we are done.
\end{proof}

\subsubsection{Ferrand pushouts of schemes}
For completeness, we compare our results to Ferrand's. This will not be used, so uninterested reader can skip to Section~\ref{pusheqsec}. By Theorem~\ref{lem:affpushispush}, Ferrand pushout of an affine scheme is affine. On the other hand, it is easy to give examples of Ferrand pushouts of schemes which are not schemes.

\begin{exam}\label{ex:nonschematicpushout}
Assume that $Z$ is a scheme, $T=\{t_1,t_2\}\subset Z$ a closed subscheme admitting no open affine neighborhood, and $f\:T\to Y$ a morphism such that $f(t_1)$ and $f(t_2)$ have a common specialization $y\in Y$. In Theorem~\ref{effth}(ii), we will establish criteria that guarantee the existence of the Ferrand pushout $X=Y\coprod_TZ$ of this kind even for compositions ($f$ is an open immersion) and pinchings ($f$ is finite). For any neighborhood $U\subseteq X$ of the image of $y$, its preimage in $Z$ contains $T$. In particular, $U$ is not affine, and hence $X$ is not a scheme.
\end{exam}

Consider the following condition ($\dag$) on a Ferrand pushout datum $\calP=(T;Y,Z)$: (a) $X$, $Y$ and $Z$ are schemes, (b) the pushout $X$ of $\calP$ in the category of locally ringed spaces is a scheme, (c) the morphism $Z\to X$ is affine, and (d) the morphism $Y\to X$ is a closed immersion. In \cite[Theorem~7.1]{Fer}, Ferrand found a necessary and sufficient condition for ($\dag$) to hold. One can also relate ($\dag$) to general Ferrand pushouts as follows:

\begin{theor}\label{schemth}
Let $\calP=(T;Y,Z)$ be a Ferrand pushout datum. Then the following conditions are equivalent:

(i) $\calP$ satisfies Ferrand's condition {\em($\dag$)}.

(ii) $\calP$ is effective in the category of algebraic spaces and $\coprod\calP$ is a scheme.

(iii) $\calP$ possesses an open affine covering.
\end{theor}
\begin{proof}
Assume (i) holds, and let $X$ be the schematic pushout of $\calP$. Since the morphism $\calP\to X$ is affine, any affine open covering of $X$ induces an open affine covering of $\calP$, and we obtain (iii). Theorem~\ref{lem:affpushispush} provides the implication (iii)$\implies$(ii). Finally, if (ii) holds then an open affine covering of $X=\coprod\calP$ induces such a covering of $\calP$, hence $X=\coprod^\Sch\calP$ by Theorem~\ref{lem:affpushispush}. Working locally it suffices to check $(\dag)$ when $\calP$ is affine, but this is covered by \cite[Theorem~5.1]{Fer}.
\end{proof}

\subsection{Pushout of equivalence relations}\label{pusheqsec}
General Ferrand pushouts are ``glued" from the schematic ones by use of \'etale equivalence relations. The following lemma plays a central role in our construction.

\begin{lem}\label{affeqrel}
Let $p_{1,2}\:\calP_1\rra\calP_0$ be an fppf equivalence relation of Ferrand push\-out data. Assume that $\calP_i$ admit open affine coverings, and set $X_i:=\coprod^\Sch\calP_i$. Then the induced morphisms $q_{1,2}\:X_1\rra X_0$ form an fppf equivalence relation.
\end{lem}
\begin{proof}
We have an fppf groupoid of pushout data $(p_1,p_2,m,i,\delta)$, where $$p_{1,2}:\calP_1\rra\calP_0, m\:\calP_2=\calP_1\times_{p_1,\calP_0,p_2}\calP_1\to\calP_1,i\:\calP_1\toisom\calP_1,\delta\:\calP_0\to\calP_1$$ satisfy all usual compatibilities, such as $p_i\circ\delta=\Id$, etc.; see \cite[\S35.11, (Tag:0231)]{stacks}. Since $\coprod^\Sch$ respects flat fiber products by Lemma~\ref{Zarlem}(ii), it takes this groupoid to an fppf groupoid of schemes $(q_1,q_2,\coprod^\Sch m,\coprod^\Sch i,\coprod^\Sch\delta)$. It remains to check that the latter groupoid is, in fact, an equivalence relation, that is, the diagonal $h\:X_1\to X_0\times X_0$ is a monomorphism.

Set $U_i:=Z_i\setminus T_i$. Then $U_i\to X_i$ is an open immersion and $|X_i|=|Y_i|\coprod|U_i|$ by Lemma~\ref{Zarlem}(i). Notice that $h$ is locally of finite type. Indeed it is the composition of $X_1\to X_1\times X_1$ and $X_1\times X_1\to X_0\times X_0$, where the first morphism is locally of finite type, and the second morphism is locally of finite presentation. Thus, by Lemma~\ref{lem:monomor}, it is sufficient to show that the base change of $h$ with respect to the surjective morphism $\left(Y_0\times Y_0\right)\coprod \left(U_0\times U_0\right)\coprod \left(U_0\times Y_0\right)\coprod \left(Y_0\times U_0\right)\to X_0\times X_0$ is a monomorphism.

The base changes of $h$ to $U_0\times Y_0$ and $Y_0\times U_0$ are monomorphisms, in fact empty morphisms, because their sources are supported on $U_1\cap Y_1$. By definition, $Y_1\rra Y_0$ and $Z_1\rra Z_0$ are equivalence relations, hence $h_Y\:Y_1\to Y_0\times Y_0$ and $h_Z\:Z_1\to Z_0\times Z_0$ are monomorphisms. Note that $$h'_Y\:Y_1\times_{X_1}Y_1=X_1\times_{(X_0\times X_0)}(Y_0\times Y_0)\to Y_0\times Y_0$$ is a base change of $h$. Since $Y_1\to X_1$ is a monomorphism (even a closed immersion), the diagonal $Y_1\to Y_1\times_{X_1}Y_1$ is an isomorphism, and hence $h_Y=h'_Y$ is the base change of $h$. Although a similar claim fails for $Z_i$, it does hold for $U_i$, and hence $h_U\:U_1\to U_0\times U_0$ is a base change of $h$. On the other hand, $h_U$ is also the base change of $h_Z$, hence a monomorphism, which completes the proof.
\end{proof}

\subsection{General pushouts}\label{generalsec}
Now, we can attack the general case.

\begin{theor}\label{genpushoutth}
Let $\calP=(T;Y,Z)$ be a Ferrand pushout datum. Then,

(i) The following conditions are equivalent: (a) $\calP$ is effective, (b) there exists an \'etale presentation $\calP=\calP_0/\calP_1$ such that both $\calP_0$ and $\calP_1$ admit open affine coverings, (c) same as (b) with an fppf presentation $\calP=\calP_0/\calP_1$.

(ii) In the situation of (c) set $X:=\coprod\calP$ and $X_i:=\coprod\calP_i$. Then $X_1\rra X_0$ is an fppf equivalence relation, $X=X_0/X_1$, and $\calP_i=\calP\times_XX_i$.
\end{theor}
\begin{proof}
(a)$\implies$(b): Pick an \'etale presentation $X_0/X_1$ of $X:=\coprod\calP$, and set $\calP_i:=\calP\times_XX_i$. Then $\calP_0/\calP_1$ is an \'etale presentation of $\calP$, and open affine coverings of $X_i$ pull back to open affine coverings of $\calP_i$ since the maps $\calP_i\to X_i$ are affine.

It remains to show that (c) implies (a) and the assertion of (ii). Let $\calP_0/\calP_1$ be an fppf presentation of $\calP$ such that $\calP_i$ possess open affine coverings. Set $X_i:=\coprod^\Sch\calP_i$ and recall that $X_i=\coprod\calP_i$ by Theorem~\ref{lem:affpushispush}. Since colimits commute, $$\coprod\calP=\coprod\Coeq(\calP_1\rra\calP_0)=\Coeq(X_1\rra X_0).$$ Furthermore, $X_1\rra X_0$ is an fppf equivalence relation by Lemma~\ref{affeqrel}, hence $X=X_0/X_1$ is the pushout of $\calP$. Finally, since by Lemma~\ref{Zarlem}(ii) $\calP_1=\calP_0\times_{X_0}X_1$ (with respect to either projection $\calP_1\to\calP_0$), it follows that $\calP_0=\calP\times_XX_0$ by flat descent and then clearly $\calP_1=\calP\times_XX_1$. In particular, the map $\calP\to X$ is affine by flat descent since $\calP_0\to X_0$ is affine, thereby proving that $\calP$ is effective.
\end{proof}


\begin{theor}\label{cor:ferprop}
Let $\calP=(T;Y,Z)$ be an effective Ferrand pushout datum, $X=\coprod\calP$, and $\phi\:\calP\to X$ the natural map. Then,

(i) The topological pushout $|Y|\coprod_{|T|}|Z|$ is naturally homeomorphic to $|X|$.

(ii) $T=Y\times_XZ$.

(iii) $Y\to X$ is a closed immersion, $U=Z\setminus T\to X$ is an open immersion, and $|X|=|Y|\coprod|U|$ set-theoretically.
\end{theor}
\begin{proof}
By Theorem~\ref{genpushoutth} there exists an \'etale presentation $\calP=\calP_0/\calP_1$ such that $\calP_i$ admit open affine coverings, and $X=X_0/X_1$ and $\calP_0=\calP\times_XX_0$, where $X_i=\coprod\calP_i$. Using flat descent this reduces all claims to the case of the pushout $\coprod\calP_0$. It remains to recall that $\coprod\calP_0=\coprod^\Sch\calP_0$ and use Lemma~\ref{Zarlem}(i).
\end{proof}

\section{Affine presentations}\label{affpressec}
Existence of $\tau$-affine coverings was important in our construction of pushouts. In particular, existence of an \'etale affine covering of $\calP$ is necessary for $\calP$ to be effective, and we will prove later that it is also sufficient. Unfortunately, we were unable to verify the existence of such coverings in general, but at least we shall establish some non-trivial criteria in this section.

\subsection{Relation to lifting problems}
A $\tau$-affine covering $\{\calP_i\to\calP\}$ is nothing but a pair of $\tau$-affine coverings $\{Y_i\to Y\}$ and $\{Z_i\to Z\}$ that agree over $T$. In many cases, the requirement that a covering of $T$ is the pull-back of a covering of $Y$ is very restrictive, e.g., this is the case if $Y$ is a point. Therefore our strategy for constructing $\{\calP_i\to\calP\}$ is always to start with a fine enough $\tau$-affine covering $\{Y_i\to Y\}$ and to lift its base change $\{T_i\to T\}$ to a $\tau$-affine covering of $Z$.

\subsubsection{Pinchings}
Pinchings are the main case where one can seriously profit from playing also with affine coverings of $Y$. This was used by J. Koll\'ar in \cite[Section~42]{Kollar} to give an elementary proof that any pinching datum possesses an \'etale affine covering (the noetherian assumption in loc.cit. is not essential and can be removed easily). A conceptual way to say that there are many \'etale coverings of $T$ coming from $Y$ is as follows.

\begin{lem}\label{pinchlem}
Assume that $f\:T\to Y$ is a finite morphism of algebraic spaces. Then \'etale coverings of the form $\{Y_i\times_YT\to T\}$, where $\{Y_i\to Y\}$ is an \'etale covering, are cofinal among all \'etale coverings of $T$.
\end{lem}
\begin{proof}
If $Y'\to Y$ is a surjective \'etale morphism, then it is sufficient to prove the lemma for the base change $f'\:T'\to Y'$ of $f$, since \'etale coverings of $T$ that factor through $T'$ form a cofinal family. Thus, we may assume that $Y$ and $T$ are schemes.

Fix an \'etale morphism $g\:U\to T$, and let $y\in g(U)$ be any point. It is sufficient to find an \'etale neighborhood $W\to Y$ of $y$ such that $W\times_YT$ factors through $U$. Let $\oY_y$ be the strict henselization of the local scheme $Y_y$. By \cite[$\rm IV_4$, 18.8.10]{ega}, $\oT=\oY_y\times_YT$ is a disjoint union of strict henselizations of $T$ at closed points, and hence $\oT\to T$ factors through $U\to T$. It remains to note that $\oY_y$ is the limit of $Y$-\'etale schemes $W_j$, hence $\oT$ is the limit of $T$-\'etale schemes $U_j=W_j\times_YT$ and by \cite[$\rm IV_3$, 8.14.2]{ega}, the $T$-morphism $\oT\to U$ factors through some $U_j$.
\end{proof}

\subsubsection{The general case}
For general Ferrand pushout data, we have only the following trivial result.

\begin{lem}\label{preslem1}
Let $\calP=(T;Y,Z)$ be a Ferrand pushout datum. If any \'etale morphism $T'\to T$ with an affine source lifts to an \'etale morphism $Z'\to Z$ with an affine source, in the sense that $T'=Z'\times_ZT$, then $\calP$ admits an \'etale affine covering.
\end{lem}
\begin{proof}
Pick an \'etale affine covering $\{Y_i\to Y\}_{i\in I}$. Then its pullback $\{T_i\to T\}$ is an \'etale affine covering that can be lifted to an \'etale family $\{Z_i\to Z\}$ such that $Z_i$ are affine schemes. Since $T\into Z$ is a closed immersion, its complement is open, and we can pick an \'etale affine covering $\{Z_j\to Z\setminus T\}_{j\in J}$. Then $\{Z_k\to Z\}_{k\in I\coprod J}$ is an \'etale affine covering of $Z$. For any $j\in J$ set $Y_j:=\emptyset$, $T_j:=\emptyset$, and for $k\in I\coprod J$ set $\calP_k:=(T_k;Y_k,Z_k)$. Then $\{\calP_k\to\calP\}$ is the required covering.
\end{proof}

\subsection{The affine lifting problem}\label{affliftsec}
Assume that $T\into Z$ is a closed immersion of algebraic spaces. By a {\em lifting} of an \'etale algebraic $T$-space $T'$ to $Z$ we mean any \'etale algebraic $Z$-space $Z'$ with $T'=Z'\times_ZT$. Lemma~\ref{preslem1} provides a motivation to raise the following affine lifting question:

\begin{question}
(AL) Assume that $T\into Z$ is a closed immersion of algebraic spaces. Given an \'etale morphism $T'\to T$ with an affine source does there exist a lifting $Z'\to Z$ with an affine source.
\end{question}

In fact, this problem is equivalent to the combination of two particular cases, a lifting question and a henselization question:

\begin{question}
Assume that $T\into Z$ is a closed immersion of algebraic spaces.

(L) If $T'$ is affine and $T'\to T$ is \'etale then does it admit a lifting to $Z$?


(H) If $T$ is affine, can one factor $T\to Z$ through an affine $Z'$ \'etale over $Z$?

\end{question}

\begin{rem}
(i) Question (H) is, closely related to the following fundamental open question about henselian schemes, cf. Conjecture B in \cite[Remark~1.23(ii)]{GS}: {\em Does there exist a non-affine henselian scheme with affine closed fiber?}

(ii) We will prove below that the affine lifting question has affirmative answer in the ind-quasi-affine case and in the case of a discrete $T$, but we do not know what happens with either (L) or (H) for more general spaces $T$ and $Z$.
\end{rem}

\begin{lem}\label{liftlem}
(AL) has affirmative answer if $Z$ is decent and $|T|$ is discrete.
\end{lem}
\begin{proof}
We shall find a lifting $Z'\to Z$ for an \'etale morphism $g\:T'\to T$ with an affine source. Clearly, we may replace $Z$ with a neighborhood $g(T')$. Thus, since $g(T')$ is quasi-compact, we may assume that so is $Z$, and hence $|T|$ is finite. Now, it is clear that it suffices to solve the lifting problem for each connected component of $T$, hence we may assume in the sequel that $|T|$ is a point.

By \cite[Tag:0ABT]{stacks}, $T$ is decent, and hence a scheme by \cite[Tag:047Z]{stacks}. Let $z$ be the reduction of $T$, then $z\to Z$ is a closed Zariski point. By \cite[Tag:0BBP]{stacks}, there exists an elementary affine \'etale neighborhood $f\:(U,u)\to (Z,z)$ of $z$, i.e., $U$ is affine, $u=f^{-1}(z)$, and $k(z)=k(u)$. Since the diagonal component $U\into U\times_ZU$ is the only component of $U\times_ZU$ containing a preimage of $z$, it follows easily that any morphism $h\:S\to Z$ with $h(S)=z$ factors through $U$. In particular, we obtain a factorization $T\to U\to Z$. Since $T'$ is discrete, we can use the local description of \'etale morphisms \cite[$\rm IV_4$, 18.4.6]{ega}, to lift the \'etale morphism $T'\to T$ to an affine \'etale morphism $Z'\to U$. Then the composition $Z'\to U\to Z$ is as needed.
\end{proof}

\begin{theor}\label{liftth}
(AL) has affirmative answer if $Z$ is an ind-quasi-affine scheme.
\end{theor}
\begin{proof}
If $Z$ is affine then the result is known, see \cite[(Tag:04D1)]{stacks}. We shall mention that our original argument in this case was rather involved and made a use of Elkik's lifting theorem. We are grateful to J. de Jong who pointed out to us that the assertion in this case is elementary.

By quasi-compactness, if $T'$ is affine then its image is contained in a quasi-affine open subscheme of $Z$. Thus, we may assume that $Z$ is quasi-affine, and hence an open subscheme of the affine hull $\oZ$. Let $\oT$ be the schematic closure of $T$ in $\oZ$; we aware the reader that the morphism from the affine hull of $T$ to $\oT$ does not have to be an open immersion. By the affine case, the \'etale morphism $T'\to T\to\oT$ lifts to an \'etale morphism $\og\:\oZ'\to\oZ$ with an affine source. The closed set $V=\oZ'\setminus\og^{-1}(Z)$ is disjoint from $T'$, and the latter is quasi-compact. Hence there exists a basic open affine $Z'\subseteq\oZ'$ containing $T'$, which is disjoint from $V$. Thus, the \'etale morphism $Z'\to\oZ$ factors through $Z$, and $g\:Z'\to Z$ is the required lifting of $f$.
\end{proof}

\subsection{Existence of \'etale affine coverings}
Using the lifting results we have proved above, existence of \'etale affine coverings can be established in the following cases.

\begin{theor}\label{cor:prescor}
A Ferrand pushout datum $\calP=(T;Y,Z)$ possesses an \'etale affine covering $\{\calP_i\to\calP\}$ if one of the following conditions is satisfied:

(i) $Z$ is an ind-quasi-affine scheme,

(ii) $\calP$ is a pinching datum, i.e., $f\:T\to Y$ is finite.

(iii) $Z$ is decent and $|T|$ is discrete.
\end{theor}
\begin{proof}
In view of Lemma~\ref{preslem1}, we see that (i) follows from Theorem~\ref{liftth} and (iii) follows from Lemma~\ref{liftlem}. In the case of a pinching datum, choose an \'etale affine covering $\{Z'_j\to Z\}$, and use Lemma~\ref{pinchlem} to find an \'etale affine covering $\{Y_i\to Y\}$ such that its pullback $\{T_i\to T\}$ factors through the pullback $\{T'_j\to T\}$ of $\{Z'_j\to Z\}$. Then $T_i\to T'_j$ can be lifted to $Z_i\to Z'_j$ by Theorem~\ref{liftth}. Thus, by setting $\calP_i:=(T_i;Y_i,Z_i)$ we obtain the desired \'etale affine covering of $\calP$.
\end{proof}

\begin{rem}
For shortness, we used Theorem~\ref{liftth} to construct \'etale affine coverings of pinching data. This is not necessary and can be bypassed by a simple local argument that only uses Lemma~\ref{preslem1}, see \cite[Section~42]{Kollar}.
\end{rem}

\begin{cor}\label{prescor1}
Let $f\:\calP\to\tilcalP$ be a morphism of Ferrand pushout data with an ind-quasi-affine $f_Z\:Z\to\tilZ$. If $\tilcalP$ admits an \'etale (resp. flat) affine covering then so does $\calP$.
\end{cor}
\begin{proof}
Let $\{\tilcalP_i\to\tilcalP\}$ be an \'etale (resp. flat) affine covering of $\tilcalP$. Set $\calP_i:=\calP\times_\tilcalP\tilcalP_i$. Then $\{\calP_i\to\calP\}$ is an \'etale (resp. flat) covering. Since $f_Z$ is ind-quasi-affine, so are its base changes $Z_i\to \tilZ_i$, and hence also $Z_i$. By Theorem~\ref{cor:prescor}(i), all $\calP_i$ admit \'etale affine coverings $\{\calP_{ij}\to\calP_i\}_j$. Hence $\{\calP_{ij}\to\calP\}_{ij}$ is the desired covering.
\end{proof}

\section{Properties of Ferrand pushouts}\label{lastsec}
In this section we study properties of spaces and morphisms preserved by the Ferrand pushout functor.

\subsection{Absolute properties}
We start with properties of spaces.

\begin{theor}\label{septh}
Let $\calP$ be an effective Ferrand pushout datum, and set $X:=\coprod\calP$. Let $(\dag)$ be any of the following properties: (1) affine, (2) possesses an open affine covering, (3) separated, (4) quasi-compact, (5) quasi-separated. Then $\calP$ satisfies $(\dag)$ if and only if so does $X$.
\end{theor}
\begin{proof}
If $X$ satisfies $(\dag)$ then so does $\calP$ since $\calP\to X$ is affine. Now, the assertion of the theorem for properties (1) and (2) follows from Theorem~\ref{lem:affpushispush}, and for property (4) the assertion is clear since $Y\coprod Z\to X$ is surjective by Theorem~\ref{cor:ferprop}.

(5) Assume that $\calP$ is quasi-separated, and let $X_i,\calP_i$ be as in Theorem~\ref{genpushoutth}. In particular, $X=X_1/X_0$, $\calP_i=\calP\times_XX_i$, and $W_i=W\times_XX_i$, where $W:=Y\coprod Z$ and $W_i:=Y_i\coprod Z_i$. We shall show that the diagonal $h\:X_1\to X_0\times X_0$ is quasi-compact. The morphism $W\to X$ is affine, hence quasi-compact, and it is surjective by Theorem~\ref{cor:ferprop}. Thus, it suffices to show the quasi-compactness of the base change $h\times_XW$. Since limits are preserved under base changes, $h\times_XW$ is the morphism $W_1=W_0\times_WW_0\to W_0\times W_0$. And the latter is quasi-compact since $W$ is quasi-separated.

(3) Assume that $\calP=(T;Y,Z)$ is separated. Then $X$ is quasi-separated by (5). Consider a valuation ring $R$ with fraction field $K$, and let $i\:\eta=\Spec(K)\to X$ be a morphism. By the valuative criterion of separatedness \cite[Tag:03KV]{stacks}, it is sufficient to show that if $f,g\:S=\Spec(R)\to X$ extend $i$ then $f=g$. Recall that $|X|=|Y|\coprod|U|$, where $U=Z\setminus T$, and $Y\to X$ is a closed immersion by Theorem~\ref{cor:ferprop}(iii). If $i(\eta)\in Y$ then the morphisms $f,g$ factor through $Y$ since $Y\to X$ is a closed immersion, and hence $f=g$ by the separatedness of $Y$. So, we may assume that $i(\eta)\in U$.

Since $f^{-1}(Y)$ and $g^{-1}(Y)$ are closed subschemes of $S$, one of them contains the other (we use here the fact that the ideals of $R$ are linearly ordered with respect to inclusion). Without loss of generality we assume that $f^{-1}(Y)$ is the larger subscheme and denote its generic point by $\veps$. Note that $\veps\neq\eta$, hence the localization $S_\veps$ is a valuation ring of non-zero height. We claim that $f(\veps)\in T$ and $g(\veps)\in T$. Pick an \'etale affine covering $\{\calP_i\to\calP\}$, and set $\calP':=\coprod_i\calP_i$, $X':=\coprod^\Sch\calP'$. Then $X'=\coprod\calP'$ by Theorem~\ref{lem:affpushispush}, and $\calP'=\calP\times_XX'$ by Theorem~\ref{genpushoutth}(ii). Set $S':=S_\veps\times_{f,X}X'$, and $S'':=S_\veps\times_{g,X}X'$. Both $S'$ and $S''$ are disjoint unions of affine \'etale schemes over $S_\veps$, hence disjoint unions of spectra of Pr\"ufer rings. By our construction, the preimages of $Y'$ in both $S'$ and $S''$ are contained in the preimage of $\veps$. Hence the morphisms $S'\to X'$ and $S''\to X'$ factor uniquely through $Z'$ by Corollary~\ref{valcor}. In particular, $f(\veps)\in T$ and $g(\veps)\in T$.

We proved that $f$ and $g$ take $S_\veps$ to $Z$, hence their restrictions onto $S_\veps$ coincide by the separatedness of $Z$. Let $W$ be the closed subscheme of $S$ with generic point $\veps$. Since $f$ and $g$ take $W$ to the separated subspace $Y\into X$ and coincide on the generic point, it follows that $f|_W=g|_W$. It remains to note that $S$ is the composition of $S_\veps$ and $W$ along $\veps$. So, $W\coprod_{\veps}S_\veps=S$ by Theorem~\ref{lem:affpushispush}, and hence $f=g$.
\end{proof}

\subsection{Criterion of effectivity}
Theorem~\ref{genpushoutth}(i) provides a criterion of effectivity of Ferrand pushout data, but it is not so convenient. Our next aim is to strengthen it by removing the assumption on $\calP_1$.

\begin{theor}\label{effth}
Let $\calP$ be a Ferrand pushout datum, then

(i) $\calP$ is effective if and only if it admits an \'etale affine covering $\{\calW_i\to\calP\}$. Moreover, if $\calP_0=\coprod_i\calW_i$ and $\calP_1=\calP_0\times_\calP\calP_0$ then $\calP_1$ possesses an open affine covering and hence the assertion of Theorem~\ref{genpushoutth}(ii) holds.

(ii) $\calP$ is effective in either of the following cases: (a) $Z$ is decent and $|T|$ is discrete, (b) $T\to Y$ is finite, (c) $Z$ is ind-quasi-affine, (d) there exists a morphism $\calP\to\calP'$ such that $\calP'$ is effective, and $Z\to Z'$ is ind-quasi-affine.
\end{theor}
\begin{proof}
Assertion (ii) follows from (i), Theorem~\ref{cor:prescor}, and Corollary~\ref{prescor1}.

(i) If $\calP$ is effective then the assertion is clear since $\calP\to \coprod\calP$ is affine. Let now $\{\calW_i\to\calP\}$ be an \'etale affine covering, $\calP_0=\coprod_i\calW_i$, and $\calP_1=\calP_0\times_\calP\calP_0$. By Theorem~\ref{genpushoutth}, it suffices to prove that $\calP_1$ possesses an open affine covering. To do so, we shall show that $\calP_1$ is effective, and $X_1:=\coprod\calP_1$ is a scheme, since then pulling back an open affine covering from $X_1$ will do the trick.

We start with effectivity. Since $\calP_0\to\calP$ is ind-quasi-affine, so is its base change $\calP_1\to\calP_0$, and hence so is $Z_1$. Thus, $\calP_1$ admits an \'etale affine covering $\{\calU_j\to\calP_1\}$ by Corollary~\ref{prescor1}. Set $\calQ_0:=\cup_j\calU_j$, and $\calQ_1:=\calQ_0\times_{\calP_1}\calQ_0$. Since $\calP_1$ is separated, both $\calQ_i$ admit open affine coverings, and hence are effective. Set $Y_i:=\coprod^\Sch\calQ_i$. Then by Theorems~\ref{lem:affpushispush} and \ref{genpushoutth}(ii), $\calP_1$ is effective, $X_1=Y_0/Y_1$, and $\calQ_i=\calP_1\times_{X_1}Y_i$.

Let us now show that $X_1$ is a scheme. Set $X_0:=\coprod^\Sch\calP_0$, and consider the morphisms $Y_i\to X_0$ and $X_1\to X_0$ induced by the first projection $p_1\:\calP_1\to\calP_0$. By Lemma~\ref{Zarlem}(ii), $\calQ_i=\calP_0\times_{X_0}Y_i$ for $i=0,1$, and it follows by flat descent that $\calP_1=\calP_0\times_{X_0}X_1$. In addition, the morphisms $Y_i\to X_0$ are \'etale by Lemma~\ref{Zarlem}(iii), hence so is $X_1\to X_0$. By Theorem~\ref{septh}, the space $X_1$ is separated, and hence a scheme by \cite[Tag:082J]{stacks}.
\end{proof}

\begin{rem}
(i) Our proof of the fact that $\calP_1$ admits an open affine covering is very indirect and rather involved. We construct an \'etale affine covering, feed it into Theorem~\ref{genpushoutth}, and then apply Zariski's main theorem for algebraic spaces to show that the pushout is, in fact, a scheme. It is an interesting question whether one can show directly (and hopefully simpler) that $\calP_1$ possesses an open affine covering, and whether one can extend the theorem to fppf affine coverings.

(ii) We do not know if this criterion can be made more explicit. We do not even know if there exist non-effective data. Note that Ferrand describes in \cite[Theorem~7.1]{Fer} such a criterion in the category of schemes, which essentially reduces to the existence of an open affine covering. In particular, one obtains a source of examples which are non-effective in the category of schemes. However, many of these examples are effective in the category of algebraic spaces, e.g., Example~\ref{ex:nonschematicpushout}, and it is not clear whether one can use this to construct a Ferrand pushout datum, which is not effective in the algebraic spaces.
\end{rem}

\subsection{Descent of properties through Ferrand pushouts}

\subsubsection{Equivalence of categories}
Lemmas~\ref{affproperties} and \ref{Zarlem} extend to the case of general Ferrand pushouts. This is one of the main results on Ferrand pushouts, so we provide a more complete list of properties they respect. A morphism of Ferrand pushout data $\psi\:\calP'\to\calP$ is called a {\em pro-open immersion} if it is isomorphic to the filtered projective limit of a family of open immersions (at least {\em a priori}, this condition is stronger than the componentwise condition).

\begin{theor}\label{equivcat}
Assume that $\calP$ is an effective Ferrand pushout datum, $X=\coprod\calP$ and $\phi\:\calP\to X$ is the natural map. Let $\gtC$ and $\gtD$ be the categories of flat algebraic spaces over $X$ and of effective flat Ferrand pushout data over $\calP$, respectively. Then,

(i) The functors $\phi^{-1}\:\gtC\to\gtD$ and $\coprod\:\gtD\to\gtC$ are essentially inverse equivalences.

(ii) The functors $\phi^{-1}\:\gtC\to\gtD$ and $\coprod\:\gtD\to\gtC$ preserve the following properties of morphisms: (1) surjective, (2) quasi-compact, (3) open immersion, (4) \'etale, (5) smooth, (6) flat, (7) locally of finite type, (8) flat and locally of finite presentation, (9) finite type, (10) flat and of finite presentation, (11) pro-open immersion. In particular, $\phi^{-1}$ and $\coprod$ respect Zariski, \'etale, fppf and fpqc topologies on $\gtC$ and $\gtD$.
\end{theor}
\begin{proof}
(i) Let us prove that if $X'$ is in $\gtC$ and $\calP'=\calP\times_XX'$ then $\coprod\calP'=X'$. Assume first that $X$ is a scheme. Choose an \'etale presentation $X'=X'_0/X'_1$ and set $\calP'_i:=\calP'\times_{X'}X'_i$, obtaining the following diagram with cartesian squares:
$$
\xymatrix{
\calP'_1\ar@<0.5ex>[r]\ar@<-0.5ex>[r]\ar[d]\ar@{}[dr]|{(1)}& \calP'_0\ar[r]\ar[d]\ar@{}[dr]|{(2)}& \calP'\ar[r]\ar[d]\ar@{}[dr]|{(3)}& \calP\ar[d] & & \calP'_1\ar@<0.5ex>[r]\ar@<-0.5ex>[r]\ar[d]\ar@{}[dr]|{(1)}& \calP'_0\ar[r]\ar[d]\ar@{}[dr]|{(23)}& \calP\ar[d]\\
X'_1\ar@<0.5ex>[r]\ar@<-0.5ex>[r]& X'_0\ar[r]& X'\ar[r]& X& & X'_1\ar@<0.5ex>[r]\ar@<-0.5ex>[r]& X'_0\ar[r]& X.
}
$$
Applying Lemma~\ref{Zarlem}(ii) to the right diagram we obtain that $X'_i=\coprod\calP'_i$ and hence $\coprod\calP'=\coprod(\calP'_0/\calP'_1)=X'_0/X'_1=X'$ by Theorem~\ref{genpushoutth}(ii).

In general, find a presentation $X=X_0/X_1$, and set $X'_i:=X'\times_XX_i$, $\calP_i:=\calP\times_XX_i$ and $\calP'_i:=\calP'\times_XX_i$. Then $\coprod\calP'_i=X'_i$ by the above case and it follows by a simple diagram chase and compatibility of colimits that $\coprod\calP'=X'$.

Vice versa, let us prove that if $\calP'$ is in $\calD$ and $X'=\coprod\calP'$ then $\calP'=\calP\times_XX'$. We will distinguish two special cases: (a) $\calP$ has an open affine covering, (b) $\calP'$ has an open affine covering. First, let us reduce the general case to case (b). Pick an \'etale affine covering $\{\calW_i\to\calP'\}$, and set $\calP'_0:=\coprod_i\calW_i$. Then $\calP'_1:=\calP'_0\times_{\calP'}\calP'_0$ admits an open affine covering by Theorem~\ref{effth}(i). By Theorem~\ref{genpushoutth}(ii), setting $X'_i:=\coprod\calP'_i$, we obtain a commutative diagram as above for which $X'=X'_0/X'_1$ and squares (1), (2) are cartesian. If the assertion is true in case (b) then (23) is also cartesian, and hence (3) is cartesian by flat descent of fiber products. Furthermore, the same argument proves the assertion in case (a) unconditionally because in the latter case (23) is cartesian by Lemma~\ref{Zarlem}(ii).

Now, let us establish case (b). By Theorem~\ref{effth}(i), there exists an \'etale presentation $\calP=\calP_0/\calP_1$, where $\calP_i$ possess open affine coverings. Then $\calP'_i:=\calP'\times_{\calP}\calP_i$ are effective by Theorem~\ref{effth}(ii)(d). Set $X_i:=\coprod\calP_i$ and $X'_i:=\coprod\calP'_i$. By case (a), $\calP_i\times_{X_i}X'_i=\calP'_i$, and by Lemma~\ref{Zarlem}(ii), $X'_i\times_{X'}\calP'=\calP'_i$. The assertion now follows by descent of fiber products.

(ii) Properties (6)--(8) are \'etale local on the target and source, hence in these cases the assertion follows by descent from Lemma~\ref{Zarlem}(iii). The assertion for properties (1), (2) is clear since the morphism $h\:Y\coprod Z\to X$ is surjective and quasi-compact, and hence a morphism of $X$-schemes is surjective (resp. quasi-compact) if and only if so is its base change with respect to $h$. A morphism $f$ satisfies (3), (4) or (5) if and only if $f$ is flat, locally of finite presentation, and its fibers satisfy the same property. Therefore, the case of (3), (4) and (5) follows from (8) and the surjectivity of $h$. Property (9) (resp. (10)) is the combination of (2) and (7) (resp. (2) and (8)). Finally, if $\calP'=\lim_i\calP_i$ for a family of open immersions $\calP_i\into\calP$ then it is also the limit in the category $\gtD$, and hence $X':=\coprod\calP'$ is the limit in $\gtC$ of the open subspaces $X_i:=\coprod\calP_i$ of $X$. Clearly, $X'=\lim_iX_i$ in the category of algebraic spaces, and hence $X'\to X$ is a pro-open immersion.
\end{proof}

\begin{cor}\label{equivcatcor}
If $\calP$ is an effective Ferrand pushout datum, and $X=\coprod\calP$, then the categories of flat $\calO_X$-modules and flat $\calO_\calP$-modules are naturally equivalent. Furthermore, the equivalence preserves the finite presentation property.
\end{cor}
\begin{proof}
Pick an \'etale affine covering $\{\calU_i\to\calP\}$, and set $\calP_0:=\coprod_i\calU_i$. Then by Theorem~\ref{effth}(i), $\calP_1:=\calP_0\times_\calP\calP_0$ possesses an open affine covering, and hence $\coprod\calP_i=\coprod^\Sch\calP_i$ by Theorem~\ref{lem:affpushispush}. Thus, the result follows by flat descent, since $X=X_0/X_1$ and $\calP=\calP_0/\calP_1$ by Theorem~\ref{genpushoutth}, and the assertion of the corollary holds true for $(\calP_i, X_i)$ by Corollary~\ref{fercor}.
\end{proof}

\subsubsection{The morphism $Z\to X$}
Our next aim is to compare properties of the morphisms $T\to Y$ and $Z\to X$ for a Ferrand pushout $X=\coprod\calP$.

\begin{theor}\label{ferdesth}
Let $X=\coprod\calP$ be a Ferrand pushout, and $(\dag)$ one of the following properties: open immersion, pro-open immersion, schematically dominant, finite, quasi-finite, finite type. Then $Z\to X$ satisfies $(\dag)$ if and only if so does $T\to Y$.
\end{theor}
\begin{proof}
Only the inverse implication needs a proof since $T\to Y$ is a base change of $Z\to X$. If $T\to Y$ is a (pro-)open immersion then $T=T\times_YT$ and hence $h\:\calP'=(T;T,Z)\to\calP$ is a morphism of pushout data. Since $h$ is a (pro-)open immersion, $Z=\coprod\calP'\to X$ is a (pro-)open immersion by Theorem~\ref{equivcat}(ii). The remaining four properties can be checked \'etale locally on $X$, so we may assume that $\calP$ and $X$ are affine by Theorems~\ref{genpushoutth} and \ref{lem:affpushispush}. The case of schematic dominance follows from \cite[Proposition~5.6(2)]{Fer} and the remaining three cases follow from \cite[Proposition~5.6(3)]{Fer}.
\end{proof}

\begin{rem}
(i) In the case of pro-open immersions, Theorem~\ref{ferdesth} actually asserts that if $X=Y\coprod_TZ$ is a composition then $Z\to X$ is a pro-open immersion. This justifies the terminology, since $X$ is ``composed'' from a closed subspace $Y$ and a pro-open subspace $Z$.

(ii) The assertion of Theorem~\ref{ferdesth} fails for $(\dag)$ being \'etale or flat even for pinchings. For example, if $Z$ is a smooth curve over a field $k$, $Y=\Spec(k)$, $T=\Spec(k\times k)\subset Z$ consists of two $k$-points, and $T\to Y$ is the projection, then $X$ is a nodal curve. So, $T\to Y$ is split \'etale, but $Z\to X$ is not even flat.
\end{rem}

\subsection{$S$-properties}
Note that even a composition of varieties over a field can be non-noetherian. In particular, the property of being of finite type over $S$ is not respected by Ferrand pushouts. Here is a list of few $S$-properties that are respected.

\begin{theor}\label{Sproperties}
Let $\calP$ be an effective Ferrand pushout datum over an algebraic space $S$, and set $X:=\coprod\calP$. Let $(\dag)$ be any of the following properties relative to $S$: (1) affine, (2) separated, (3) quasi-compact, (4) quasi-separated. Then $\calP$ satisfies $(\dag)$ if and only if so does $X$.
\end{theor}
\begin{proof}
Only the direct implication requires a proof. All properties are \'etale-local on $S$, and  if $S'\to S$ is an \'etale morphism, $X'=X\times_SS'$, and $\calP'=\calP\times_SS'=\calP\times_XX'$, then $X'=\coprod\calP'$ by Theorem~\ref{equivcat}(i). Therefore, we may assume that $S$ is affine, and the result follows from Theorem~\ref{septh}.
\end{proof}

\bibliographystyle{amsalpha}
\bibliography{nagata}

\end{document}